\documentclass[11pt]{article}

\usepackage[top=1in, bottom=1.25in, left=1.25in, right=1.25in]{geometry}
\usepackage{amsmath}
\usepackage{amsfonts}
\usepackage{amsthm}
\usepackage{amssymb}
\usepackage{thm-restate}
\usepackage{enumerate}
\usepackage{framed}
\usepackage{mathtools}
\usepackage{xifthen}
\usepackage{float}
\usepackage{subcaption}
\usepackage{chngcntr}
\usepackage[nobysame]{amsrefs}

\usepackage{MnSymbol}

\newtheorem{thm}{Theorem}
\newtheorem{lem}[thm]{Lemma}
\newtheorem{prop}[thm]{Proposition}
\newtheorem{cor}[thm]{Corollary}
\newtheorem{conj}[thm]{Conjecture}

\newtheorem{qn}[thm]{Question}
\newtheorem*{claim}{Claim}

\newtheoremstyle{primedtheorem} %
    {} % spaceabove
    {} % spacebelow
    {\itshape} % bodyfont
    {} % indent
    {\bfseries} % headfont
    {.} % headpunctuation
    {.5em} % headspace
    {#1 \thmnote{#3}'} % headspec

\theoremstyle{primedtheorem}
\newtheorem*{thmprimed}{Theorem}

\theoremstyle{definition}

\counterwithin*{equation}{thm}

\newcommand{\diagrams}[1]{#1}

\newcommand{\N}{\mathbb{N}}
\newcommand{\Z}{\mathbb{Z}}

\newcommand{\es}{\emptyset}
\newcommand{\sm}{\setminus}

\newcommand{\Zk}[1]{\Z_{#1}}

\newcommand{\Cu}{C^\text{up}}
\newcommand{\Cd}{C_\text{down}}
\newcommand{\Tu}{T^\text{up}}

\newcommand{\up}[2][]{{#2}\ifthenelse{\isempty{#1}}{'}{'^{(#1)}}}
\newcommand{\uup}[2][]{{#2}^{\ifthenelse{\isempty{#1}}%
        {\normalfont\dagger}{\normalfont\dagger(#1)}}}

\newcommand{\inZk}[1]{\overline{#1}}
\newcommand{\idbyA}[1]{#1(\alpha)}

\newcommand{\multi}[2][]{#1\cdot#2}

\newenvironment{claimproof}{%
    \par\noindent\textit{Proof of claim.}
}{%
    \hfill$\square$
}

\newcommand{\enumskip}{15pt}

\begin{document}

\nocite{*}

\title{Tiling with arbitrary tiles}

\author{%
    Vytautas Gruslys\thanks{%
        Department of Pure Mathematics and Mathematical Statistics,
        Centre for Mathematical Sciences,
        University of Cambridge,
        Wilberforce Road,
        CB3\;0WB Cambridge,
        United Kingdom;
        e-mail: \mbox{\texttt{\{v.gruslys,i.leader\}@dpmms.cam.ac.uk}}\,.
    }
    \and
    Imre Leader\footnotemark[1]
    \and
    Ta Sheng Tan\thanks{%
        Institute of Mathematical Sciences,
        Faculty of Science,
        University of Malaya,
        50603 Kuala Lumpur,
        Malaysia;
        e-mail: \texttt{tstan@um.edu.my}\,.
        This author acknowledges support received from the Ministry
        of Education of Malaysia via FRGS grant FP048-2014B.
    }
}

\maketitle

\begin{abstract}
    Let $T$ be a tile in $\mathbb{Z}^n$, meaning a finite
    subset of $\mathbb{Z}^n$. It may or may not tile $\mathbb{Z}^n$, in the
    sense of $\mathbb{Z}^n$ having a partition into copies of $T$. However,
    we prove that $T$ does tile $\mathbb{Z}^d$ for some $d$.
    This resolves a conjecture of Chalcraft.
\end{abstract}

%\maketitle

\section{Introduction}

    Let $T$ be a \emph{tile}, by which we mean a finite non-empty subset
    of $\Z^n$ for some $n$. It is natural to ask if $\Z^n$ can
    be partitioned into copies of $T$, that is, into subsets each of
    which is isometric to $T$. If such a partition exists, we
    say that $T$ \emph{tiles} $\Z^n$.

    For instance, consider the following tiling of $\Z^2$ by copies of
    the $C$-shaped pentomino.

    \begin{figure}[H]
        \centering
        \includegraphics[scale=0.4]{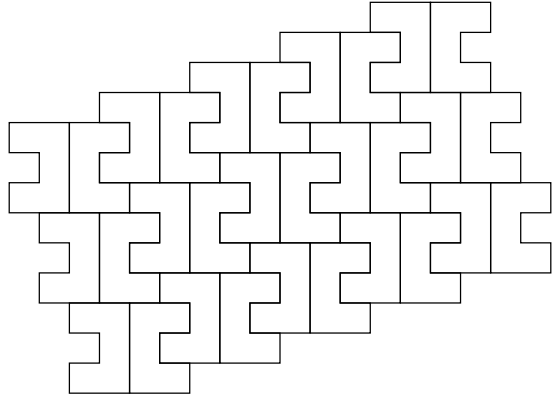}
        \caption{
            The $C$-shaped pentomino tiles $\Z^2$.
        } 
        \label{fig:tetromino-s}
    \end{figure}

    As another example, the one-dimensional tile $\mathtt{X.X}$ (to 
    be understood as $\{1, 3\}$) tiles $\Z$, and so does $\mathtt{XX.X}$\,.
    On the other hand, $\mathtt{XX.XX}$ is a one-dimensional tile that
    does not tile $\Z$. Does it tile some space of higher dimension? The
    following diagram shows that $\mathtt{XX.XX}$ does tile $\Z^2$.

    \begin{figure}[H]
        \centering
        \includegraphics{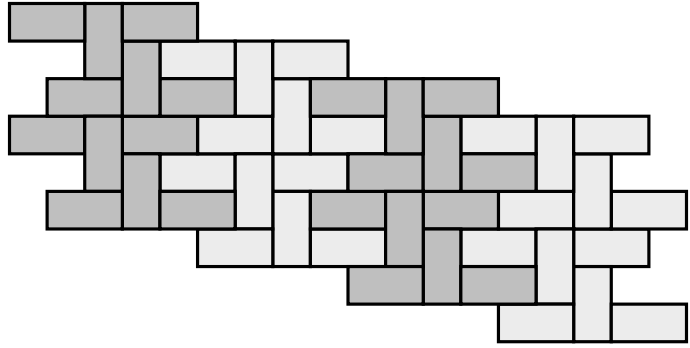}
        \caption{
            This pattern is formed from disjoint copies of $\mathtt{XX.XX}$;
            copies of the pattern may be stacked vertically to tile $\Z^2$.
        } 
        \label{fig:xx.xx}
    \end{figure}

    A similar pattern works for $\mathtt{XXX.XX}$ in $\Z^2$. However,
    one can check by hand that $\mathtt{XXX.XXX}$ does not tile $\Z^2$.
    Does it tile $\Z^3$, or $\Z^d$ for some $d$? What about more
    complicated one-dimensional tiles?

    Let us now consider a couple of two-dimensional examples. Let $T$ denote
    the $3 \times 3$ square with the central point removed. Clearly $T$
    does not tile $\Z^2$, since the hole in a copy of $T$ cannot be
    filled. However, in $\Z^3$ there is enough space for one copy of $T$
    to fill the hole of another. (Of course, this in no way implies that
    $T$ does tile $\Z^3$.)

    For a `worse' example, consider the $5 \times 5$ square with
    the central point removed. Two copies of such tile cannot be
    interlinked in $\Z^3$. However, there is, of course, enough space
    in $\Z^4$ to fill the hole, as demonstrated in the following diagram.

    \begin{figure}[H]
        \centering
        \includegraphics[scale=1.6]{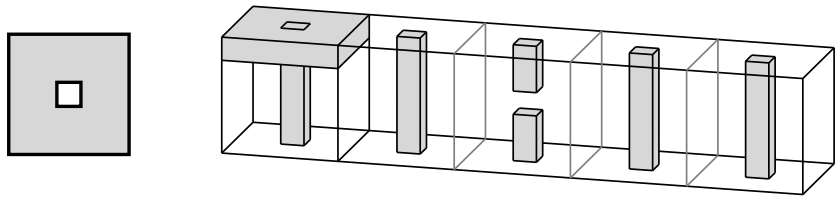}
        \caption{
            The diagram on the right is four-dimensional and shows a
            $5 \times 5 \times 5 \times 5$ region of $\Z^4$.
            Let $x_1, x_2, x_3, x_4$ be the directions of $\Z^4$.
            Each of the five $5 \times 5 \times 5$ cubes corresponds
            to a fixed value of $x_1$. Increasing the value of $x_1$ by $1$
            means jumping
            from a cube to the cube on its right. This four-dimensional
            diagram contains two copies of the two-dimensional tile
            depicted on the left side. One copy is horizontal and can
            be found in the top left part of the diagram. The second
            copy is formed by the vertical columns.
        } 
        \label{fig:coveringholes}
    \end{figure}

    Chalcraft \cite{mathforum, mathoverflow} made the remarkable conjecture
    that every tile $T \subset \Z$, or even $T \subset \Z^n$, does tile
    $\Z^d$ for some $d$.

    \begin{conj}[Chalcraft]
        Let $T \subset \Z^n$ be a tile. Then $T$ tiles $\Z^d$ for some $d$.
    \end{conj}

    It is not important if reflections are allowed
    when forming copies of a tile. Indeed, any reflection of an
    $n$-dimensional tile can be obtained by rotating it in $n+1$
    dimensions. It is also not important if only connected tiles are
    considered, as it is an easy exercise to show that any disconnected
    tile in $\Z^n$ tiles a connected tile in $\Z^{2n}$.

    In this paper we prove Chalcraft's conjecture.

    \begin{restatable}{thm}{main} \label{thm:main}
        Let $T \subset \Z^n$ be a tile. Then $T$ tiles $\Z^d$ for some $d$.
    \end{restatable}

    Interestingly, the problem is not any easier for tiles $T \subset \Z$.
    Indeed, the proof for one-dimensional tiles seems to us to be as hard
    as the general problem.

    The plan of the paper is as follows. In Section~\ref{sec:simplecase}
    we prove a special case of the theorem, namely when $T$ is an
    interval in $\Z$ with one point removed. The aim of this
    section is to demonstrate some of the key ideas in a simple setting.
    The proof of the general case builds on these ideas and on
    several additional ingredients. We give a proof
    of Theorem~\ref{thm:main} in Section~\ref{sec:generalcase}.
    Finally, in Section~\ref{sec:conclusion} we give some open problems.

    We end the section with some general background.
    A lot of work has been done about tiling $\Z^2$ by polyominoes
    (a polyomino being a connected tile in $\Z^2$). Golomb \cite{golomb66}
    proved that every polyomino of size at most $6$ tiles $\Z^2$.
    In \cite{golomb70} he also proved that there is no algorithm which
    decides, given a finite set of polyominoes, if $\Z^2$ can be tiled
    with their copies -- this is based on the work of Berger \cite{berger66},
    who showed
    a similar undecidability result for Wang tiles (which are certain coloured
    squares). However, it is not known if such an algorithm exists for single
    polyominoes. A related unsolved problem is to determine whether there
    is a polyomino which tiles $\Z^2$ but such that every tiling is
    non-periodic. On the other hand, Wijshoff and van Leeuwen
    \cite{wijshoff84} found an algorithm which determines if disjoint
    translates (rather than translates, rotations and reflections) of
    a single given polyomino tile $\Z^2$. A vast number of results and
    questions regarding tilings of $\Z^2$ by polyominoes and other shapes
    are compiled in Gr\"unbaum and Shephard \cite{grunbaum13}.

    One may also wish to know if a given polyomino tiles some finite region
    of $\Z^2$, say a rectangle. This class of questions has also received
    significant attention, producing many beautiful techniques and invariants
    -- see, for example, \cite{conway90, golomb89, klarner69}. In the context
    of this paper, we observe that there are tiles which cannot tile any
    (finite) cuboid of any dimension. For example, consider the plus-shaped
    tile of size $5$ in $\Z^2$: this tile cannot cover the corners of any
    cuboid. In fact, there are one-dimensional such tiles. For example,
    let $T \subset \N$, where $\N = \{1, 2, \dotsc\}$, be a symmetric tile
    (meaning that $-T$ is a translate of $T$) whose associated
    polynomial $p(x) = \sum_{t \in T} x^t$ does not have all of its non-zero
    roots on the unit circle -- it turns out that such $T$ cannot tile a
    cuboid (see \cite{mathoverflow}).

    \section{Tiling $\Z^d$ by an interval minus a single point}
    \label{sec:simplecase}

    \subsection{Overview}
    \label{subsec:overview}

        Before starting the proof of Theorem~\ref{thm:main}, we demonstrate
        some of the key ideas in a simple setting, where the tile is a
        one-dimensional interval with one point removed. We give a
        self-contained proof of the general case in
        Section~\ref{sec:generalcase}, but it will build on the ideas in this
        section.

        We write $[k] = \{1, \dotsc, k\}$.

        \begin{restatable}{thm}{simpletheorem} \label{thm:simplecase}
            Fix integers $k \ge 3$ and $i \in \{2, \dotsc, k-1\}$ and
            let $T$ be the tile $[k] \sm \{i\}$. Then $T$ tiles $\Z^d$
            for some $d$.
        \end{restatable}
        
        The tile $T = [k] \sm \{i\}$ will remain fixed throughout
        this section.

        The proof is driven by two key ideas. A first natural idea is to
        use \emph{strings}, where a string is a one-dimensional infinite line
        in $\Z^d$ with every $k$-th point removed. Note that any string is a
        union of disjoint copies of $T$. An obvious way to use strings
        would be to partition $\Z^d$ into them. Although this is an
        attractive idea, it is not possible, for the following simple reason:
        if we consider just the fixed cuboid $[k]^d \subset \Z^d$, then 
        every string intersects it in exactly $0$ or $k-1$
        points, but the order of $[k]^d$ is not divisible by $k-1$.

        This suggests a refinement of the idea. We will try to use strings
        parallel to $d-1$ of the $d$ directions, while the remaining
        direction will be special and
        copies of $T$ parallel to it will be used even without forming
        strings. In other words, we will view $\Z^d$ as $\Z \times \Z^{d-1}$,
        that is, as being partitioned into $(d-1)$-dimensional \emph{slices}
        according to the value of the first coordinate. We will first
        put down some tiles parallel to the first direction (each such tile
        intersects multiple slices), and then complete the tilings in each
        slice separately by strings.

        To do this we need another idea. What subsets of $\Z^{d-1}$
        can be tiled by strings? Note that a partial tiling of $\Z^{d-1}$
        by strings can be identified with a partial tiling of the discrete
        torus $\Zk{k}^{d-1}$ (where $\Zk{k}$ denotes the integers modulo $k$),
        where a tile in $\Zk{k}^{d-1}$ means any line with one point
        removed.  The size of $\Zk{k}^{d-1}$ is $k^{d-1}
        \equiv 1 \pmod{k-1}$, so any such partial tiling of $\Zk{k}^{d-1}$
        must leave out $1 \pmod{k-1}$ points. Of course, it is far
        from true that any subset of $\Zk{k}^{d-1}$ of size a multiple
        of $k-1$ may be partitioned into tiles. However, our plan is
        to find a large supply of sets that do have this property. In
        particular, it turns out that a key idea will be to find
        a large set $C \subset \Zk{k}^{d-1}$
        such that for any choice of distinct elements
        $x_1, \dotsc, x_m \in C$ with $m \equiv 1 \pmod{k-1}$, $T$ does tile
        $\Zk{k}^{d-1} \sm \{x_1, \dotsc, x_m\}$.

        \begin{figure}[H]
            \centering
            \includegraphics[scale=0.8]{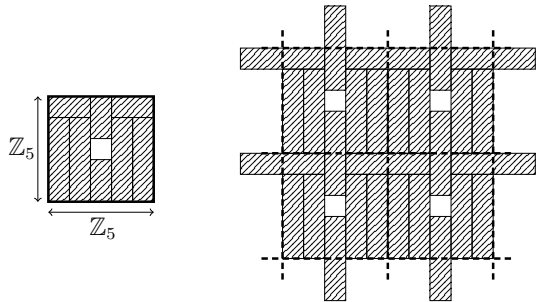}
            \caption{
                A partial tiling of $\Z_k^2$ (here $k=5$) corresponds to a
                partial tiling of $\Z^2$ by strings.
            } 
            \label{fig:torustoperiodic}
        \end{figure}

        These ideas work together as follows (see Figure~%
        \ref{fig:specialdimension}). First, in $\Z \times
        \Zk{k}^{d-1}$ (for large $d$) we find a subset $X$ which is a disjoint
        union of translates of $T \times \{0\}^{d-1}$ and has the property
        that for any $n \in \Z$ the set $\{ x \in \Zk{k}^{d-1} : (n, x) \in
        X\}$ is a subset of $C$ of size congruent to $1$ modulo $k-1$.
        Then $T$ tiles $(\{n\} \times \Zk{k}^{d-1}) \sm X$. This holds for
        all $n \in \Z$, so in fact $T$ tiles $\Z \times \Zk{k}^{d-1}$, and hence
        it tiles $\Z^d$, establishing Theorem~\ref{thm:simplecase}.

        \begin{figure}[ht]
            \centering
            \includegraphics[scale=0.8]{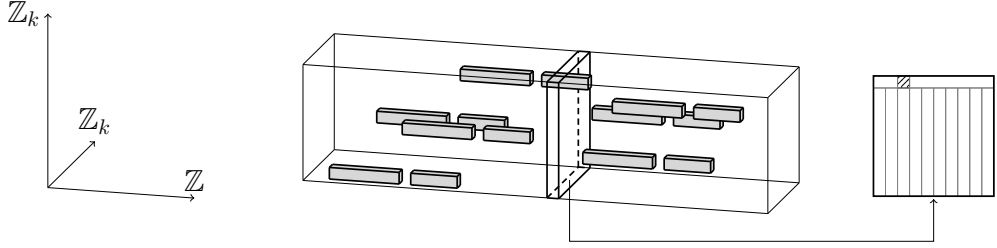}
            \caption{
                The aim is to put down tiles parallel to one of the directions
                so that the remainder of each slice could be tiled
                by strings. This diagram only symbolically visualises
                this principle. In particular, the slices here are
                two-dimensional, while in the proof they can have much
                higher dimension.
            } 
            \label{fig:specialdimension}
        \end{figure}

        The rest of this section is organised as follows.
        In Section~\ref{subsec:removedcorners-simple} we consider partial
        tilings by strings.
        In Section~\ref{subsec:specialdimension-simple} we consider the
        special direction.
        Both ideas are combined in Section~\ref{subsec:formalities-simple},
        where a full proof of Theorem~\ref{thm:simplecase} is given.

    \subsection{Tiling $\Zk{k}^d$ with some elements removed}
    \label{subsec:removedcorners-simple}

        For any $1 \le j \le d$, define the \emph{$j$-th corner}
        of $\Zk{k}^d$ to be $c_{j,d}$ where
        $$
            c_{j,d} =
                    (
                        \underbrace{%
                            0,\dotsc,0,
                            \overset{%
                                \overset{%
                                    \mathclap{j\text{-th coordinate}}
                                }{%
                                    \downarrow
                                }
                            }{%
                                k-1
                            },
                            0,\dotsc,0
                        }_{d \text{ coordinates}}
                    )
            \in \Zk{k}^d.
        $$
        Write $C_d = \{c_{j,d} : j = 1, \dotsc, d\}$ for the set of corners.

        \begin{figure}[H]
            \centering
            \includegraphics{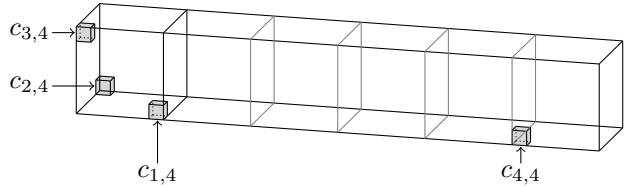}
            \caption{
                The set of corners $C_4$ when $k = 6$. In this diagram
                the space $\Zk{6}^4 = \{(x_1, x_2, x_3, x_4) : x_i \in
                \{0, \dots, 5\}\}$ is split from left to right, according
                to the value of $x_4$, into $6$ three-dimensional slices.
            } 
            \label{fig:corners}
        \end{figure}

        Looking ahead, our aim later will be to provide some copies of $T$
        in the $x_1$-direction in $\Z \times \Zk{k}^d$, at heights
        corresponding to points of $C_d$, and in such a way that what
        remains in each $\Zk{k}^d$ can be partitioned into lines with one
        point removed. But first we need to create a useful supply of
        such subsets of $\Zk{k}^d$.

        Recall that $|\Zk{k}^d| \equiv 1\pmod{k-1}$, so if $T$ tiles some
        set $X \subset \Zk{k}^d$ (here and in the remainder
        of this section $T$ is identified with its image under
        the projection $\Z \to \Zk{k}$, so its copies in $\Zk{k}^d$ are lines
        with one point removed), then $|\Zk{k}^d \sm X|
        \equiv 1\pmod{k-1}$. In this section we will prove
        Lemma~\ref{lem:removedcorners}, which is an approximate converse
        of this statement.

        \begin{lem} \label{lem:removedcorners}
            Let $d \ge 1$ and suppose that $S \subset C_d$ is such
            that $|S| \equiv 1 \pmod{k-1}$ and $|S| \le
            d - \log_k{d}$. Then $T$ tiles $\Zk{k}^d \sm S$.
        \end{lem}

        In fact, this lemma holds even without the assumption that
        $|S| \le d - \log_k{d}$, but we keep it for the sake of simpler
        presentation.

        We will prove Lemma~\ref{lem:removedcorners} at the end of this
        section. Meanwhile, we collect the tools needed for the proof.
        In fact, there are several ways to prove Lemma~%
        \ref{lem:removedcorners}. The method outlined here is quite general,
        and we will build on it in Section \ref{sec:generalcase}.

        We start with a simple proposition.  

        \begin{prop} \label{prop:coverbutone}
            Let $d \ge 1$ and $x \in \Zk{k}^d$. Then
            $T$ tiles $\Zk{k}^d \sm \{x\}$.
        \end{prop}

        \begin{proof}[Proof (see Figure~\ref{fig:onemissing})]
            Use induction on $d$. If $d = 1$, then $\Zk{k}\sm\{x\}$ is
            itself a translate of $T$. Now suppose that $d \ge 2$ and
            write $x = (x_1, \dotsc, x_d)$, $\hat{x} = (x_1, \dotsc, x_{d-1})$.
            By the induction hypothesis, for each $j \in \Zk{k}$,
            $(\Zk{k}^{d-1}\sm\{\hat{x}\}) \times \{j\}$ can be tiled with
            copies of $T$.  It remains to tile $\{\hat{x}\} \times (\Zk{k}
            \sm \{x_d\})$, but this is itself a copy of $T$.
        \end{proof}

        \begin{figure}[H]
            \centering
            \includegraphics{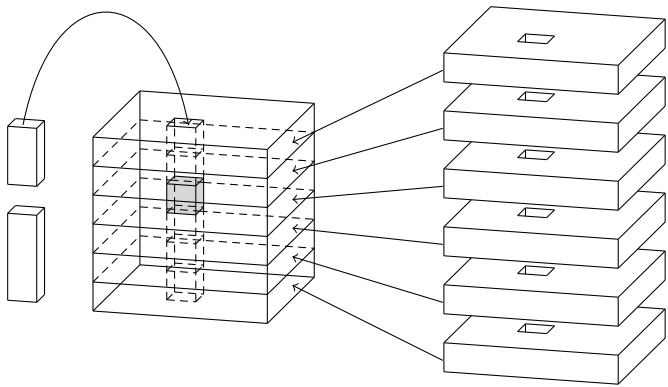}
            \caption{
                The induction step in the proof
                of Proposition~\ref{prop:coverbutone}. The grey cube
                represents $x$.  The vertical column in which $x$ lies,
                without $x$ itself, is a copy of $T$. Each horizontal slice
                minus the
                point in this column can be tiled by the induction hypothesis.
            } 
            \label{fig:onemissing}
        \end{figure}

        Let $X \subset \Zk{k}^d$ (for any $d \ge 1$) be such that
        $T$ tiles $\Zk{k}^d \sm X$. We will say that such $X$ is a
        \emph{hole} in $\Zk{k}^d$.
        The intuition for $X$ is that it
        is a set that remains uncovered after an attempt to tile
        $\Zk{k}^d$ by copies of $T$.

        We can identify $X$ with a higher-dimensional set
        $\up{X} = X \times \{0\} \subset \Zk{k}^{d+1}$.
        One can easily verify that $\up{X}$ is a hole in
        $\Zk{k}^{d+1}$. More importantly, we will show in the
        following proposition that a single additional point of $X'$ can
        be covered in exchange for leaving the $(d+1)$-st corner of
        $\Zk{k}^{d+1}$ uncovered (see Figure~\ref{fig:movepoint}). 
        This is why, for any $S \subset \Zk{k}^d$, we define
        $$
            \uup{S} = (S \times \{0\}) \cup \{ c_{d+1, d+1} \}
                    \subset \Zk{k}^{d+1}.
        $$
        Note that the definition of $\uup{S}$ and the definition of $S$
        being a hole depend not only on $S$, but also on the dimension of the
        underlying discrete torus $\Zk{k}^d$. For $m \ge 1$, we will
        use the shorthand $\uup[m]{S}$ to denote the result of $m$
        consecutive applications of the $\uup{}$ operation to $S$, that is,
        $$
            \uup[m]{S} = S{\underbrace{^{\dagger\dotsc\dagger}}_{m}}
                       \subset \Zk{k}^{d+m}.
        $$

        \begin{figure}[H]
            \centering
            \includegraphics[scale=0.5]{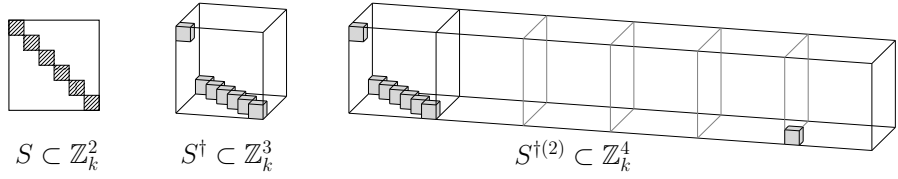}
            \caption{
                Suppose $S$ is the subset of $\Zk{k}^2$ given in
                the diagram on the left (here $k=6$). The diagram in the
                middle depicts $\uup{S}$, and the diagram on
                the right depicts $\uup[2]{S}$. Observe that $S$ is
                a hole in $\Zk{k}^2$, but $\uup{S}$ and $\uup[2]{S}$
                are not holes in $\Zk{k}^3$ and $\Zk{k}^4$, respectively.
            } 
            \label{fig:dagger}
        \end{figure}

        \begin{prop} \label{prop:movepoint}
            Let $d \ge 1$ and let $X \subset \Zk{k}^d$ be a hole.
            Then for each $x \in X$ the set $\uup{(X \sm \{x\})}$ is a
            hole in $\Zk{k}^{d+1}$.
        \end{prop}

        \begin{figure}[H]
            \centering
            \includegraphics[scale=0.5]{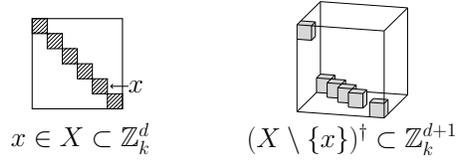}
            \caption{
                An illustration of the statement of Proposition~%
                \ref{prop:movepoint}. The aim is to show that $T$ tiles
                $\Zk{k}^{d+1} \setminus (X \setminus \{x\})^{\dagger}$.
            } 
            \label{fig:movepoint}
        \end{figure}

        \begin{proof}[Proof (see Figure~\ref{fig:movepointproof})]
            Use (i) a tiling of $\Zk{k}^d \sm X$ for $(\Zk{k}^d \sm X) \times
            \{0\}$, and (ii) one copy of $T$ to cover
            $\{x\}\times(\Zk{k}\sm\{k-1\})$. By Proposition~%
            \ref{prop:coverbutone}, (iii) $(\Zk{k}^d\sm\{(0,\dotsc,0)\})\times
            \{k-1\}$ and (iv) $(\Zk{k}^d\sm\{x\})\times\{i\}$, $i \in
            \{1,\dotsc,k-2\}$, can each be tiled by copies of $T$. 
        \end{proof}

        \begin{figure}[H]
            \centering
            \includegraphics[trim=0 5mm 0 0]{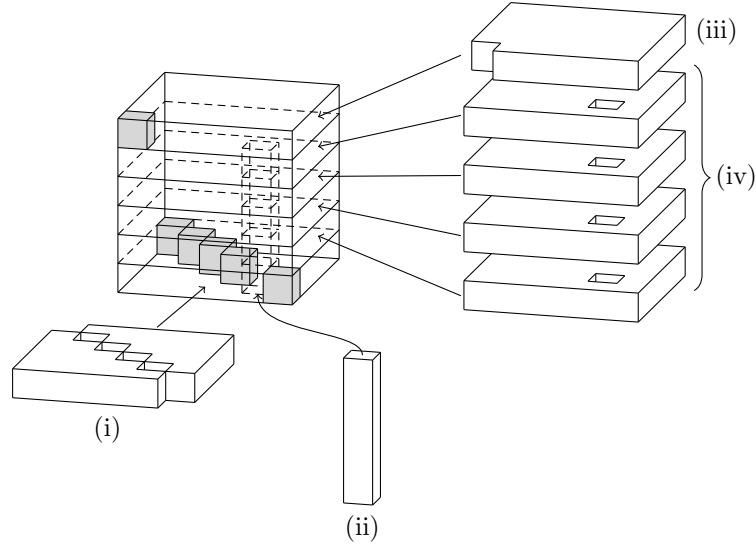}
            \caption{
                The bottom horizontal piece (i) is tilable because $X$ is a
                hole, and the other horizontal pieces (iii) and (iv) are tilable
                by Proposition~\ref{prop:coverbutone}.  The remaining vertical
                column (ii) is a copy of $T$.
            } 
            \label{fig:movepointproof}
        \end{figure}

        We will apply Proposition~\ref{prop:movepoint} inductively,
        that is, in the form of the following corollary.

        \begin{cor} \label{cor:exchangeall-simple}
            Let $d \ge 1$ and let $X \subset \Zk{k}^d$ be a hole.
            Then for any distinct elements $x_1, \dotsc x_m \in X$, the
            set $\uup[m]{(X \sm \{x_1, \dotsc, x_m\})}$ is a hole
            in $\Zk{k}^{d+m}$. \qed
        \end{cor}
        
        We are now ready to prove Lemma \ref{lem:removedcorners}.

        \begin{proof}[Proof of Lemma \ref{lem:removedcorners}]
            Write $|S| = m$ and $r = d-m$. By symmetry, we can assume
            that
            $$
                S =
                    \left\{ c_{j,d} : j = r+1, \dotsc, d \right\}.
            $$

            Our aim is to prove that $S$ is a hole in $\Zk{k}^d$. Note
            that $S = \uup[m]{\es}$, where the empty set $\es$ is
            considered as a subset of $\Zk{k}^r$. Therefore by
            Corollary~\ref{cor:exchangeall-simple} it suffices to find
            a hole $X \subset \Zk{k}^r$ with $|X| = m$.

            This can be done by partitioning $\Zk{k}^r$ into a singleton
            $\{x\}$ and copies of $T$ (this can be done by
            Proposition~\ref{prop:coverbutone}),
            and letting $X$ be the union of $\{x\}$ and the appropriate
            number of copies of $T$.  By assumption, $m \equiv 1\pmod{k-1}$
            so the only potential problem with this construction of $X$ is
            if $|\Zk{k}^r| < m$. However, this is ruled out by the assumption
            that $m \le d - \log_k{d}$.
        \end{proof}

    \subsection{Using one special direction to get $T$-tilable slices}
    \label{subsec:specialdimension-simple}

        The purpose of this section is to demonstrate that tiles in the first
        direction in $\Z \times \Zk{k}^{d-1}$ (that is, translates of
        $T \times \{0\}^{d-1}$)
        can be combined in such a way that the uncovered part of each slice
        can be tiled by copies of $T$ using Lemma~\ref{lem:removedcorners}.
        The exact claim is as follows.

        \begin{lem} \label{lem:specialdimension}
            There exists a number $\ell \ge 1$ such that for any $d \ge 1$ and
            any set
            $C \subset \Zk{k}^{d-1}$ of
            order $|C| \ge \ell$ there is a set $X \subset \Z \times C$,
            satisfying:
            \begin{enumerate}[\hspace{\enumskip}(a)]
                \item $X$ is a union of disjoint sets of the form
                    $(T + n) \times \{c\}$ with $n \in \Z$ and $c
                    \in C$;

                \item $|(\{n\} \times C) \cap X| \equiv 1%
                    \pmod{k-1}$ for every $n \in \Z$;
                \item $|(\{n\} \times C) \cap X| \leq \ell$ for every $n \in
                    \Z$.
            \end{enumerate}
        \end{lem}

        \begin{figure}[H]
            \centering
            \includegraphics{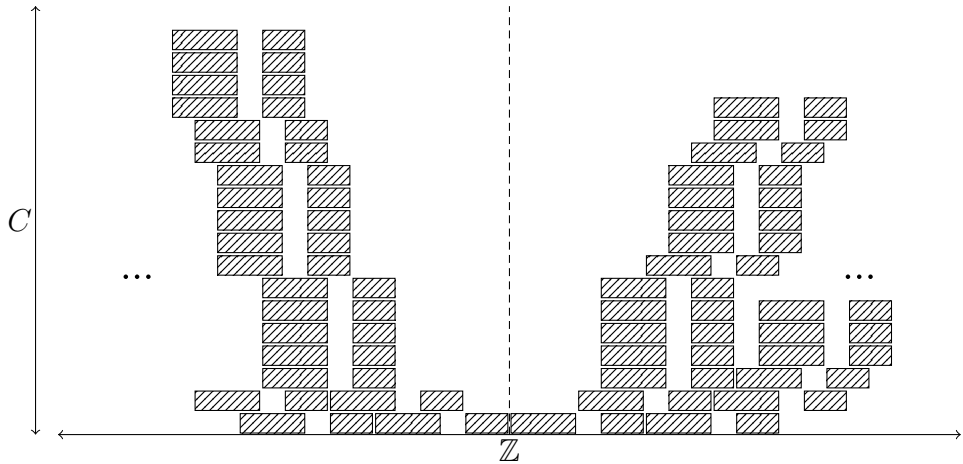}
            \caption{
                A possible construction of $X$. In this
                example the aim is to have $1$ modulo $6$ elements covered 
                in each column.
            } 
            \label{fig:specialdimensionin2d}
        \end{figure}

        We start with the following trivial proposition.

        \begin{prop}
        \label{prop:deconvolution-simple}
            There is a function $f : \Z \to \{0, \dots, k-2\}$ such
            that for each $x \in \Z$
            $$
                \sum_{y \in T} f(x - y) \equiv 1 \pmod{k-1}.
            $$
        \end{prop}

        \begin{proof}
            Start by defining $f(n) = 0$ for $-k+1 \le n \le -1$. Now
            define $f(n)$ for $n \ge 0$ as follows. Suppose that for
            some $n \ge 0$ the values of $f(j)$ are already defined for all
            $j$ such that $-k+1 \le j \le n-1$. Then the value of $f(n)$ is
            uniquely defined by
            $$
                f(n) \hspace{5pt} \equiv \hspace{5pt}
                    1 \hspace{5pt} - \sum_{y \in T \sm \{1\}} f(n+1 - y)
                    \hspace{5pt} (\text{mod }k-1).
            $$

            Define $f(n)$ for all $n \le -k$ in a similar way.
        \end{proof}

            Now Lemma~\ref{lem:specialdimension} can be proved quickly.

        \begin{proof}[Proof of Lemma~\ref{lem:specialdimension}]
            
            Write $\ell = 2k(k-2)$ and suppose that $|C| = \ell$. 
            Let $f : \Z \to \{0, \dots, k-2\}$ be as given
            by Proposition~\ref{prop:deconvolution-simple}.
            The aim is to choose subsets $S_n \subset C$ for
            every $n \in \Z$, with orders satisfying $|S_n| = f(n)$,
            and such that $S_m \cap S_n = \es$ whenever $m \neq n$ and
            $(T+m)\cap(T+n) \neq \es$.  Then $X$ can be taken to be
            $\bigcup_{n\in\Z} (T+n) \times S_n$.

            Fix any enumeration of $\Z$, and define the sets $S_n$ one by one in
            that order. When defining $S_n$, there can be at most $2k - 1$
            choices of $m$ with $S_m$ already defined and $m-n \in T-T$.
            Moreover, $|S_m| \le k-2$ for each $m$. Therefore to be able to find
            $S_n$ it is enough to have $|C| - (2k - 1)(k-2) \ge f(n)$. Finally,
            this condition is ensured by the choice of $\ell$, completing the
            proof in the case when $|C| = \ell$. If $|C| > \ell$, we are done by
            restricting to a subset of $C$ of size exactly $\ell$.
        \end{proof}

    \subsection{Completing the proof of Theorem~\ref{thm:simplecase}}
    \label{subsec:formalities-simple}

        It was noted in Section~\ref{subsec:overview} that Lemmas~%
        \ref{lem:removedcorners} and \ref{lem:specialdimension} together
        imply that for some $d \ge 1$
        \begin{align}
            &T \text{ tiles } \Z\times\Zk{k}^{d-1},
                \label{eq:zzkdtilable}\\
            \shortintertext{and therefore}
            &T \text{ tiles } \Z^d,
                \label{eq:zdtilable}
        \end{align}
        implying Theorem~\ref{thm:simplecase}. However, some abuse of
        notation is already present in the statement of
        (\ref{eq:zzkdtilable}). In this section we will carefully
        explain what is meant by (\ref{eq:zzkdtilable}), why it follows
        from the two lemmas and how it implies (\ref{eq:zdtilable}).
        In doing so, we will complete the proof of Theorem~%
        \ref{thm:simplecase}.

        To avoid confusion, within this section we use quite precise
        language. Although this might seem pedantic here, for later it will be
        very important to have precise notation available. We denote the
        elements of $\Zk{k}$ by $\inZk{x}$ for
        $x \in \Z$ (instead of identifying them with $x$, which was our
        preferred notation in the rest of the section), and we will denote
        the image of $T$ under the natural projection $\pi : \Z \to \Zk{k}$ by
        $\pi(T)$ rather than simply by $T$.

        \simpletheorem*

        \begin{proof}
            Fix a large $d$ (more precisely, first let $\ell$ be as
            given by Lemma~\ref{lem:specialdimension} and then fix $d$
            such that $d-1 - \log_k (d-1) \ge \ell$).

            Denote the projection map $\Z\to\Zk{k}$ by $\pi$, and
            consider the following subsets of $\Z \times \Zk{k}^{d-1}$:

            \newcommand{\zz}{\,\left\{\,\inZk{0}\,\right\}\,}
            \begin{alignat*}{5}
                \mathsf{T}_1 &=
                    \hspace{6pt} T && \times \,\zz && \times \zz
                    && \times \dotsb && \times \zz, \\
                \mathsf{T}_2 &=
                \big\{\,0\,\big\} && \times \pi(T) && \times \zz 
                    && \times \dotsb && \times \zz, \\
                &\hspace{6pt}\vdots && && && && \\
                \mathsf{T}_d &=
                \big\{\,0\,\big\} && \times \,\zz && \times \zz
                    && \times \dotsb && \times \pi(T).
            \end{alignat*}

            Recall from Section~\ref{subsec:removedcorners-simple} the
            definition of
            $$
                C_{d-1} =
                    \left\{%
                        (\,
                            \underbrace{%
                                \inZk{0}\,,\,\dotsc\,,\,\inZk{0}\,,
                                \overset{%
                                    \overset{%
                                        \mathclap{j\text{-th coordinate}}
                                    }{%
                                        \downarrow
                                    }
                                }{%
                                    \,\inZk{k-1}\,
                                },
                                \,\inZk{0}\,,\,\dotsc\,,\,\inZk{0}
                            }_{d-1 \text{ coordinates}}
                        \,) : j = 1,\dotsc,d-1
                    \right\}
                \subset \Zk{k}^{d-1}.
            $$
            By Lemma~\ref{lem:specialdimension}, there
            is a set $X \subset \Z \times C_{d-1}$, which is a union
            of disjoint translates of $\mathsf{T}_1$ and for
            each $n \in \Z$ satisfies
            $|(\{n\}  \times C_{d-1}) \cap X| \le d-1 - \log_k (d-1)$ and
            $|(\{n\}  \times C_{d-1}) \cap X| \equiv 1\pmod{k-1}$. Hence,
            by Lemma~\ref{lem:removedcorners}, $(\{n\}\times\Zk{k}^{d-1})
            \sm X$ is a union of disjoint translates of
            $\mathsf{T}_2, \dotsc, \mathsf{T}_d$ for each
            $n \in \Z$.
            Therefore $\Z \times \Zk{k}^{d-1}$ is a union of
            disjoint translates of $\mathsf{T}_1, \dotsc,
            \mathsf{T}_d$ (this is exactly what is meant by
            (\ref{eq:zzkdtilable})).

            More explicitly, there are integers $1 \le \idbyA{t}
            \le d$ and $\idbyA{x_1}, \dotsc, \idbyA{x_d} \in \Z$,
            indexed by $\alpha\in A$, such that $\Z\times\Zk{k}^{d-1}$
            is the disjoint union
            $$
                \Z\times\Zk{k}^{d-1} = \bigsqcup_{\alpha\in A}
                    \Big[
                        \mathsf{T}_{\idbyA{t}}+
                        \big(
                            \,\idbyA{x_1},
                            \,\inZk{\idbyA{x_2}}\,,
                            \,\dotsc\,,
                            \,\inZk{\idbyA{x_d}}\,
                        \big)
                    \Big].
            $$

            From this it follows that, in fact, $\Z^d$ is $T$-tilable.
            Indeed, consider the following subsets of $\Z^d$:
            \begin{alignat*}{5}
                \mathsf{T}_1' &=
                    \hspace{4pt}T && \times \{0\} && \times \{0\}
                    && \times \dotsb && \times \{0\}, \\
                \mathsf{T}_2' &=
                    \{0\} && \times \hspace{4pt}T && \times \{0\}
                    && \times \dotsb && \times \{0\}, \\
                &\hspace{6pt}\vdots && && && && \\
                \mathsf{T}_d' &=
                    \{0\} && \times \{0\} && \times \{0\}
                    && \times \dotsb && \times \hspace{4pt}T.
            \end{alignat*}
            Then we can express $\Z^d$ as the disjoint union
            \[ 
                \Z^d =
                    \bigsqcup_{
                        \substack{
                            \alpha\in A \\
                            \mathclap{c_2,\dotsc,c_d\in\Z}
                        }
                    }
                \Big[
                    \mathsf{T}_{\idbyA{t}}'+
                        \big(
                            \idbyA{x_1},
                            \,\idbyA{x_2}+kc_2\,,
                            \,\dotsc\,,
                            \,\idbyA{x_d}+kc_d\,
                        \big)
                \Big]. \qedhere
            \]
        \end{proof}

\section{The general case}
\label{sec:generalcase}

    Recall the statement of the main theorem.
    
    \main*

    In this section we prove the main theorem by generalising
    the approach demonstrated in Section~\ref{sec:simplecase}.
    We have to account for two ways in which Theorem~%
    \ref{thm:simplecase} is a special case: firstly, the tile
    can be multidimensional; secondly, even in the one-dimensional
    case the tile can have more complicated structure than in
    Section~\ref{sec:simplecase}.

    It turns out that dealing with the first issue does
    not add significant extra difficulty to the proof,
    provided that the right setting is chosen. Namely,
    most of the intermediate results will be stated in
    terms of abelian groups rather than integer lattices.
    This way a multidimensional tile $T \subset \Z^b$ can be
    considered as being one-dimensional, if $\Z^b$ (rather than $\Z$)
    is chosen as the underlying abelian group. Moreover, this
    point of view is vital for comparing periodic tilings
    of an integer lattice with tilings of a discrete torus,
    already an important idea in the proof of the special case.

    On the other hand, dealing with the second issue requires
    significant effort. It involves finding the right way to
    generalise the two key ideas from Section~\ref{sec:simplecase},
    as well as introducing a new ingredient that allows the
    argument to be applied iteratively.

    We now introduce some definitions.
    Given an abelian group $G$, we call any non-empty
    subset $T \subset G$ a \emph{tile} in $G$. Given
    abelian groups $G_1,\dotsc,G_d$ and corresponding
    tiles $T_i\subset G_i$, consider the following subsets
    of $G_1 \times \dotsb \times G_d$:
    \begin{alignat*}{4}
        \mathsf{T}_1 &= \hspace{3pt} T_1 &&\times \{0\}
            &&\times \dotsb &&\times \{0\}, \\
        \mathsf{T}_2 &= \{0\} &&\times \hspace{3pt} T_2
            &&\times \dotsb &&\times \{0\}, \\
        &\hspace{6pt}\vdots && && && \\
        \mathsf{T}_d &= \{0\} &&\times \{0\}
            &&\times \dotsb &&\times \hspace{3pt} T_d.
    \end{alignat*}
    Any translate of such $\mathsf{T}_i$ (that
    is, a set of the form $\mathsf{T}_i + x$ for
    $x \in G_1\times\dotsb\times G_d$) is called a
    \emph{copy of $T_i$}. We say that a subset
    $X\subset G_1\times\dotsb\times G_d$ is
    \emph{$(T_1,\dotsc,T_d)$-tilable} if $X$ is
    a disjoint union of copies of $T_1,\dotsc,T_d$.

    It will often be the case that $(G_1, T_1) = \dotsb =
    (G_d, T_d) = (G, T)$. Then we will use the term
    \emph{$T$-tilable} as a shorthand for $(T,\dotsc,T)$-%
    tilable.     

    More generally, we may consider subsets of $G_1^{d_1}
    \times \dotsb \times G_m^{d_m}$ where $G_1, \dotsc,
    G_m$ are abelian groups with tiles $T_i \subset G_i$.
    In this setting we would say that a subset is
    $(\multi[d_1]{T_1}, \dotsc, \multi[d_m]{T_m})$-tilable.
    In other words, each $\multi[d_i]{T_i}$ replaces
    $$
    \underbrace{T_i, \dotsc, T_i}_{d_i}.
    $$
    However, we suppress ``$\multi[1]{}$'' in the notation.
    So, for example, we could say that a subset of
    $G_1^7 \times G_2 \times G_3^{10}$ is $(\multi[7]{T_1},
    T_2, \multi[10]{T_3})$-tilable.

    \subsection{A summary of the proof}
    \label{subsec:summary}

        Let $T$ be a fixed finite tile in $\Z^b$. Without
        loss of generality assume that $T \subset [k]^b$
        for some $k \ge 1$. Then, writing $\pi:\Z^b\to\Zk{k}^b$
        for the projection map, $\pi(T)$ is a tile in
        $G = \Zk{k}^b$.

        In the light of the argument from Section~\ref{sec:simplecase},
        one might hope to find a positive integer $d$ and a large
        family $\mathcal{F}$ of disjoint subsets of $G^d$ with the
        property that whenever a subfamily $\mathcal{S} \subset \mathcal{F}$
        with $|\mathcal{S}| \equiv 1 \pmod{|T|}$ is chosen, the
        set $G^d \sm (\bigcup_{S\in\mathcal{S}}S)$ is $\pi(T)$-tilable.
        However, this seems to be achievable only in the case when $\pi(T)$
        is in a certain sense a `dense' subset of $G$.

        If $\pi(T)$ is sparse, we achieve a weaker aim. Namely,
        we find a certain set $X \subset G^d$ which has sufficiently nice
        structure and is a denser subset of $G^d$ than $\pi(T)$ is of $G$.
        Also, we find a large family $\mathcal{F}$ of disjoint subsets of $X$
        such that for any $\mathcal{S} \subset \mathcal{F}$ of appropriate
        size $X \sm (\bigcup_{S \in \mathcal{S}} S)$ is $\pi(T)$-tilable.
        Taking copies of $T$ in the special direction, we can now
        tile $\Z^b \times X$.

        Repeating this process, we can use copies of $T$ and $\Z^b \times X$
        to tile $\Z^p \times Y$ for an even denser subset $Y \subset
        G^l$. After finitely many iterations of this procedure we tile
        the whole of $\Z^q \times G^m$ for some possibly large
        $q$ and $m$. From this it follows that $\Z^{q + bm}$
        is $T$-tilable.

        The rest of this section is organised as follows.
        In Section~\ref{subsec:removedcorners} we show how any tile
        in a (finite) abelian group $H$ can be used to \emph{almost} tile a
        sufficiently nice denser subset of $H^d$ for some $d$. This is the most
        complicated part of the proof, but it shares a similar structure
        with the simpler argument in Section~\ref{subsec:removedcorners-%
        simple}.

        In Section~\ref{subsec:specialdimension} we show how
        one special dimension can be used to cover the gaps in every
        slice. The argument is almost identical to the one in
        Section~\ref{subsec:specialdimension-simple}.
        
        In Section~\ref{subsec:general} we observe some simple transitivity
        properties of tilings. They enable the
        iterative application of the process. The ideas in this section
        are fairly straightforward.
        
        Finally, in Section~\ref{subsec:formalproof} we compile the tools
        together and complete the proof of Theorem~\ref{thm:main}.

    \subsection{Almost tiling denser multidimensional sets}
    \label{subsec:removedcorners}
        
        Our goal is to prove the following lemma.

        \begin{restatable}{lem}{constructdenser}
        \label{lem:constructdenser}

            Let $T \subsetneq G$ be a tile in a finite abelian
            group $G$. Then there is a set $A \subset G$,
            with $T \subsetneq A$, having the following
            property. Given any $d_0 \ge 1$, there is some
            $d \ge d_0$ and a family $\mathcal{F}$ consisting of
            at least $d_0$ pairwise disjoint subsets of $A^d$
            such that
            $$
                G \times
                        \left(
                            A^d \sm
                            \bigcup_{S \in \mathcal{S}} S
                        \right)
                \subset
                    G^{d+1}
            $$
            is $T$-tilable whenever $\mathcal{S} \subset
            \mathcal{F}$ satisfies $|\mathcal{S}| \equiv 1%
            \pmod{|T|}$.

        \end{restatable}

        Before presenting the proof, we make a few definitions
        that will hold throughout this section. First, let
        $G$ and $T$ be fixed as in the statement of Lemma~%
        \ref{lem:constructdenser}. Since $T \neq G$, we can
        fix an $x \in G$ such that $T + x \neq T$. Define
        \begin{align*}
            &\Tu = T+x,\\
            &\Cu = \Tu\sm T,\\
            &\Cd = T\sm\Tu,\\
            &A = T\cup\Tu
        \end{align*}
        (see Figures~\ref{fig:definitions}~and~\ref{fig:fourdimensions}).

        \begin{figure}[h]
            \centering
            \includegraphics[scale=0.7]{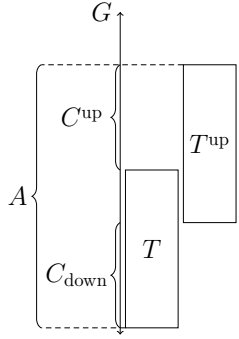}
            \caption{
                An illustration of the definitions.
            }
            \label{fig:definitions}
        \end{figure}

        \begin{figure}[h]
            \centering
            \includegraphics[scale=0.8]{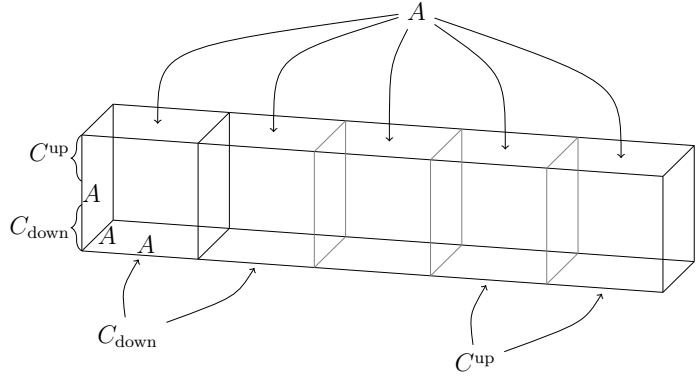}
            \caption{
                A four-dimensional diagram of $A^4$. The sets $\Cd$
                and $\Cu$ are marked on two of the axes. In this example
                $|A| = 5$ and $|\Cd| = |\Cu| = 2$.  This and the following
                four-dimensional diagrams in this section
                should be understood more generally as depicting $A^d$
                for any $d$, the three-dimensional slices representing copies
                of $A^{d-1}$.
            } 
            \label{fig:fourdimensions}
        \end{figure}

        We will use $A$ from this definition in the proof of 
        Lemma~\ref{lem:constructdenser}. For the family $\mathcal{F}$
        we will take all sets of the following form.  For any integers
        $1\le i\le d$, write
        $$
            C_{i,d}=%
                \underbrace{%
                    \Cd\times\dotsb\times\Cd\times
                    \overset{%
                        \overset{%
                            \mathclap{i\text{-th component}}
                        }{%
                            \downarrow
                        }
                    }{%
                        \Cu
                    }\times
                    \Cd\times\dotsb\times\Cd
                }_{d \text{ components}}%
            \subset A^d
        $$
        (see Figure~\ref{fig:cornersinfourdimensions}).
        Also write $C_{0,d} = (\Cd)^d$.  Note that if $i \neq j$,
        then $C_{i,d} \cap C_{j,d} = \es$.

        Finally, as $T$ is fixed, we can simply say \emph{tilable}
        instead of $T$-tilable.  

        \begin{figure}[h]
            \centering
            \includegraphics[scale=0.8]{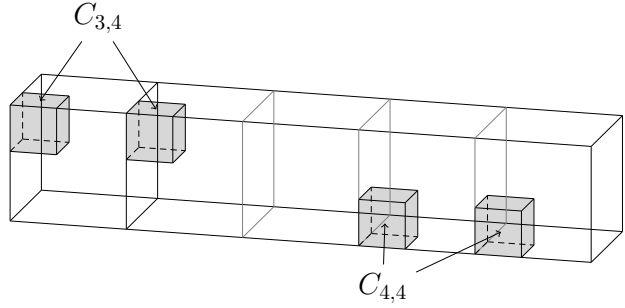}
            \caption{
                A four-dimensional diagram, which extends the previous
                diagram. Note that $C_{3,4}$ and $C_{4,4}$ both intersect 
                two three-dimensional slices, because in
                this example $|\Cd| = |\Cu| = 2$.
            } 
            \label{fig:cornersinfourdimensions}
        \end{figure}

        One of the reasons why these definitions are useful
        is that they allow the following analogue of Proposition~%
        \ref{prop:coverbutone}.

        \begin{prop} \label{prop:removedcset}
            For any integers $d \ge 1$ and $0 \le i \le d$,
            the set $A^d \sm C_{i,d}$ is tilable.
        \end{prop}

        \begin{proof}[Proof (see Figure~\ref{fig:removedcset})]
            Use induction on $d$.
            If $d=1$, observe that $A = \Cu \sqcup T
            = \Cd \sqcup \Tu$, and so $A\sm C_{i,1}$ ($=
            A\sm\Cu \text{ or } A\sm\Cd$) is a translate
            of $T$.

            Now suppose that $d\ge 2$ and without loss
            of generality assume that $i \neq d$. By the
            induction hypothesis, for each
            $g \in A$, the slice $(A^{d-1}\sm C_{i,d-1})
            \times \{g\}$ can be $T$-tiled. It remains
            to tile the set $C_{i,d-1}\times(A\sm \Cd)=
            C_{i,d-1}\times\Tu$,
            but this is obviously a union of disjoint
            copies of $T$.
        \end{proof}

        \begin{figure}[h]
            \centering
            \includegraphics[scale=0.8]{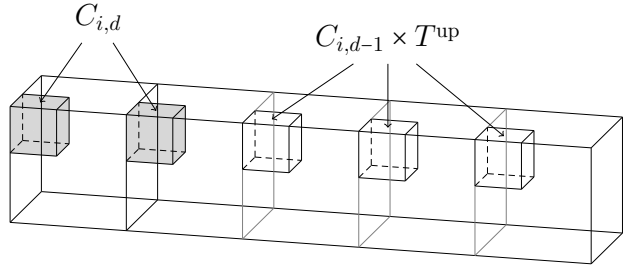}
            \caption{
                The induction step in the
                proof of Proposition~\ref{prop:removedcset}. The
                set $C_{i,d-1} \times \Tu$ is a union of copies of
                $\Tu$. In each slice it remains to tile a copy of
                $A^{d-1} \sm C_{i,d-1}$. This can be done by the
                induction hypothesis.
            } 
            \label{fig:removedcset}
        \end{figure}

        We now make a series of definitions that
        are useful for lifting subsets of lower-dimensional
        spaces to higher-dimensional spaces.

        A \emph{basic set} is a set of the form $A^d,\,
        G \times A^d$ or $\{g\} \times A^d$ for some $g \in G$, with $d$ any
        positive integer. Let $X$ be a subset of a basic set $\Omega$
        and write $\Omega = W \times A^d$ (so $W = A^0,\,G$ or $\{g\}$
        for some $g \in G$). We define 
        $$
        \uup{X} = (X \times \Cd) \cup (W \times C_{d+1,d+1})
                    \subset W \times A^{d+1}
        $$
        (see Figure~\ref{fig:dagger-general}).

        \begin{figure}[h]
            \centering
            \includegraphics{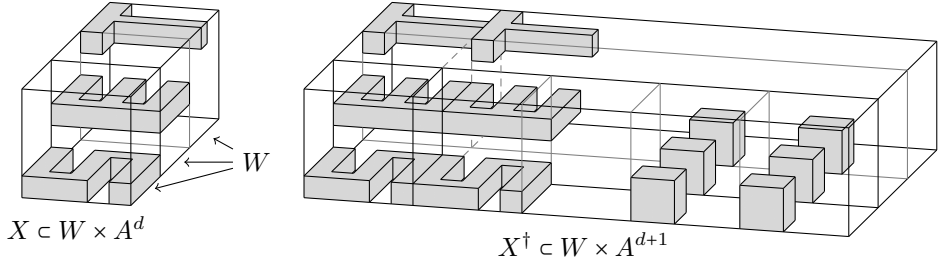}
            \caption{
                An illustration of the definition of $\uup{X}$, building
                on Figure~\ref{fig:fourdimensions}. The
                diagram on the left is four-dimensional and represents
                a generic set $X \subset W \times A^d$. The diagram
                on the right is five-dimensional and represents the
                corresponding $\uup{X}$.
                We stress that this is an abstract illustration. In
                particular, here  $|A| = 5$ and $|W| = 3$, while in fact
                we always have either $|W| = 1$ or $|W| = |G| \ge |A|$.
            } 
            \label{fig:dagger-general}
        \end{figure}

        Moreover, for
        any $m \ge 1$ we use the shorthand $\uup[m]{X}$ to denote the
        result of $m$ consecutive applications of the $\uup{}$
        operation to $X$, that is,
        \begin{align*}
            \uup[m]{X} &= X{\underbrace{^{\dagger\dotsc\dagger}}_{m}} \\
                       &= \big(
                           X \times C_{0,m}
                          \big) 
                            \cup
                          \left(
                             W \times C_{d+1,d+m} 
                          \right) 
                            \cup
                          \dotsb
                            \cup
                          \left(
                             W \times C_{d+m,d+m}
                          \right) \\
                       &\subset W \times A^{d+m}.
        \end{align*}

        For the final definition, we say that $X$ is a
        \emph{hole} in $\Omega$ if $\Omega \sm X$ is tilable.  
        Note that these definitions depend not only on $X$, but
        also on the underlying basic set $\Omega$. Therefore we will
        only use them when the underlying set is explicitly stated
        or clear from the context.

        \begin{prop} \label{prop:usecset}
            Let $d \ge 1$ and let $X$ be a hole in $A^d$.
            Suppose that $C_{i,d} \subset X$
            for some $0 \le i \le d$. Then $\uup{(X \sm C_{i,d})}$
            is a hole in $A^{d+1}$.
        \end{prop}

        \begin{proof}[Proof (see Figure~\ref{fig:usecset})]
            Partition $A^{d+1} \sm \uup{(X \sm C_{i,d})}$
            into four sets
            \begin{enumerate}[\hspace{\enumskip}(i)]
                \item $C_{i,d} \times (A \sm \Cu)$ --- tilable,
                    because $A \sm \Cu = T$;
                \item $(A^d \sm X) \times \Cd$ --- tilable, because
                    $A^d \sm X$ is tilable;
                \item $(A^d \sm C_{i,d}) \times (A \sm (\Cu \cup \Cd))$
                    --- tilable by Proposition~\ref{prop:removedcset};
                \item $(A^d \sm C_{0,d}) \times \Cu$
                    --- tilable by Proposition~\ref{prop:removedcset}.
                    \qedhere
            \end{enumerate}
        \end{proof}

        \begin{figure}[h]
            \centering
            \includegraphics{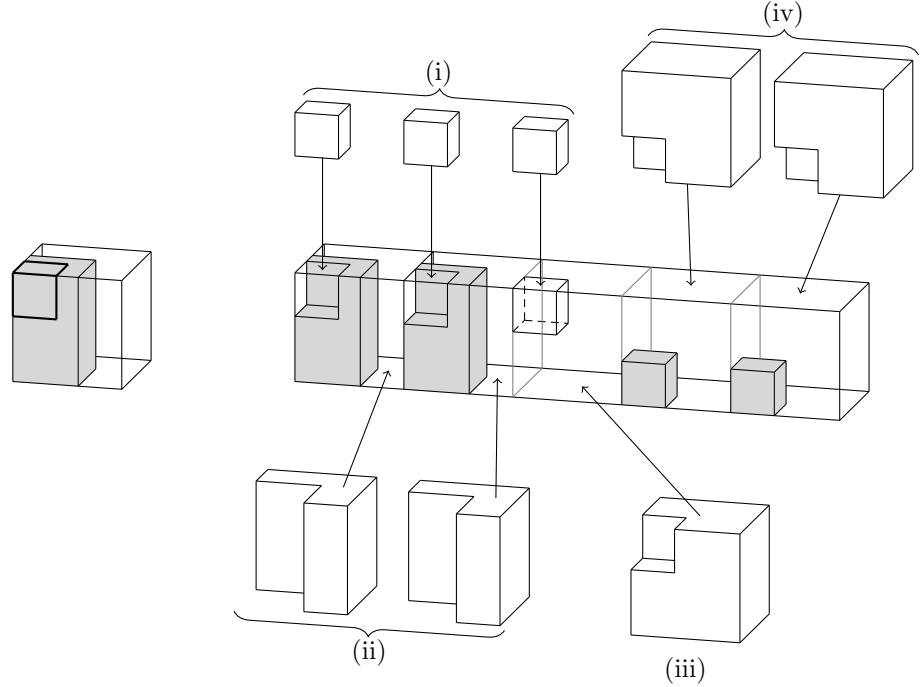}
            \caption{
                An illustration of the proof of Proposition~%
                \ref{prop:usecset}. The three-dimensional diagram on
                the left represents a hole $X \subset A^d$ which
                contains $C_{i,d}$. The four-dimensional diagram
                on the right represents $\uup{(X \sm C_{i,d})}$ and
                demonstrates why it is a hole in $A^{d+1}$.
            } 
            \label{fig:usecset}
        \end{figure}

        This proposition is the most useful for us in the form of the
        following corollary.

        \begin{cor} \label{cor:mcsets}
            Let $d \ge 1$ and suppose that $0 \le i_1, \dotsc, i_m
            \le d$ are distinct integers. Then
            $$
                \uup[m]{%
                    \big(
                        A^d \sm (C_{i_1, d} \cup \dotsb \cup C_{i_m, d})
                    \big)
                }
            $$
            is a hole in $A^{d+m}$.
        \end{cor}

        \begin{proof}
            Use induction on $m$. The base case $m=1$ is a special
            case of Proposition~\ref{prop:usecset}, so suppose
            that $m \ge 2$. Note that
            \begin{align*} 
                &\uup[m]{%
                    \big(
                        A^d \sm
                        (C_{i_1,d} \cup \dotsb \cup C_{i_m,d})
                    \big)} \\
                =& 
                \uup{
                    \left(
                        \uup[m-1]{%
                            \big(
                                A^d \sm
                                (C_{i_1,d} \cup \dotsb \cup C_{i_{m-1},d})
                            \big)
                        }
                        \sm C_{i_m,d+m-1}
                    \right)
                }
            \end{align*}
            so it is a hole in $A^{d+m-1}$ by the induction
            hypothesis and Proposition~\ref{prop:usecset}.
        \end{proof}

        Now we have the tools needed for the proof of Lemma~%
        \ref{lem:constructdenser}.

        \begin{proof}[Proof of Lemma~\ref{lem:constructdenser}]

            Fix any $d \ge (1 + |G|/|T|)d_0$ and write
            $\mathcal{F} = \{C_{i,d} : i = 1, \dotsc, d\}$.
            By symmetry, it is enough to find a tiling for
            the set
            $$
                M_m =
                G \times
                    \left(
                        A^d \sm
                            (
                                C_{d-m+1,d} \cup \dotsb \cup C_{d,d}
                            )
                    \right)
            $$
            for every choice of $m \le d_0$ with
            $m \equiv 1\pmod{|T|}$. Fix one such value of $m$,
            and let $M = M_m$ be the corresponding set that we
            have to tile.
            
            Define $r = d - m$ and $\Omega = G \times A^r$.
            We will construct a partition
            $\mathcal{B}$ of the set $\Omega$, satisfying:
            \begin{itemize}
                \item $\mathcal{B}$ consists of the set $Y_0 =
                    G \times C_{0,r}$ and copies of the tile $T$;
                
                \item for each $1 \le i \le r$, there is
                    some $y_i \in G$ such that the set $Y_i =
                    (T+y_i)\times C_{i,r}$ is exactly the union of
                    some copies of $T$ in $\mathcal{B}$; 

                \item each $y \in G$ appears at least
                    $t= (m-1) / |T|$ times in the list
                    $y_1, \dotsc, y_r$.
            \end{itemize}

            \begin{figure}[H]
                \centering
                \includegraphics[width=0.9\textwidth]{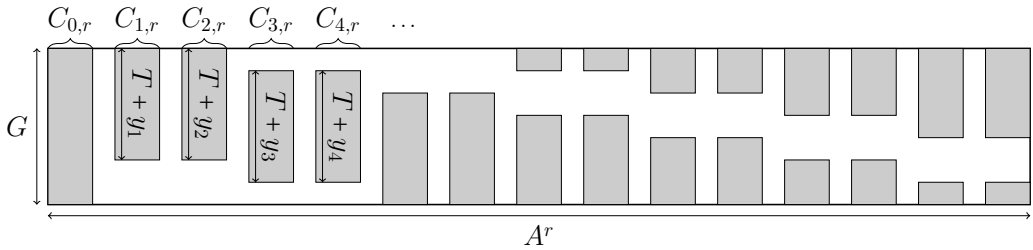}
                \caption{
                    By constructing the partition $\mathcal{B}$ we show
                    that the set $\bigcup_{i=0}^r Y_i$ (grey in this
                    diagram) is a hole in $\Omega = G \times A^r$. In fact,
                    $\bigcup_{i\in I\cup\{0\}} Y_i$ is a hole for any
                    $I \subset [r]$.
                } 
                \label{fig:blueprint}
            \end{figure}

            We start the construction by fixing any list
            $y_1,\dotsc, y_r$ such that each member of $G$ appears
            exactly $t$ times in $y_1, \dotsc, y_{t|G|}$ (in
            particular, this list satisfies the final condition displayed
            above).  Note that such a list exists since
            $r \ge t |G|$.
            
            Now we use induction to construct, for each $0 \le j \le r$,
            a partition $\mathcal{B}_j$ of $G \times A^j$ such that
            the first two conditions are satisfied when $\mathcal{B}$ and
            $r$ are replaced by $\mathcal{B}_j$ and $j$.

            Let $\mathcal{B}_0 = \{G\}$. Having defined
            $\mathcal{B}_{j-1}$, let $\mathcal{B}_j$ consist of the
            following sets (see Figure~\ref{fig:blueprintconstruction}):
            \begin{enumerate}[\hspace{\enumskip}(i)]
                \item $G\times C_{0,j}$,
                \item $X \times \{b\}$ for each $X \in
                    \mathcal{B}_{j-1}$ that is a copy of $T$
                    and each $b \in \Cd$,
                \item $\{g\} \times \{a\} \times \Tu$
                    for each $g \in G \sm (T + y_j)$ and
                    each $a \in A^{j-1}$,
                \item $(T+y_j) \times \{a\} \times \{b\}$
                    for each $a \in A^{j-1}$ and each
                    $b \in \Tu = A\sm \Cd$.
            \end{enumerate}
            One can easily check that $\mathcal{B}_j$ is a
            partition of $G \times A^j$ with the required properties. In
            particular, the sets of the first two types cover
            $G \times A^{j-1} \times \Cd$, and the remaining sets cover
            $G \times A^{j-1} \times (A \sm \Cd)$.

            This concludes the construction of $\mathcal{B}$.

            \begin{figure}[H]
                \centering
                \includegraphics[width=.9\textwidth]{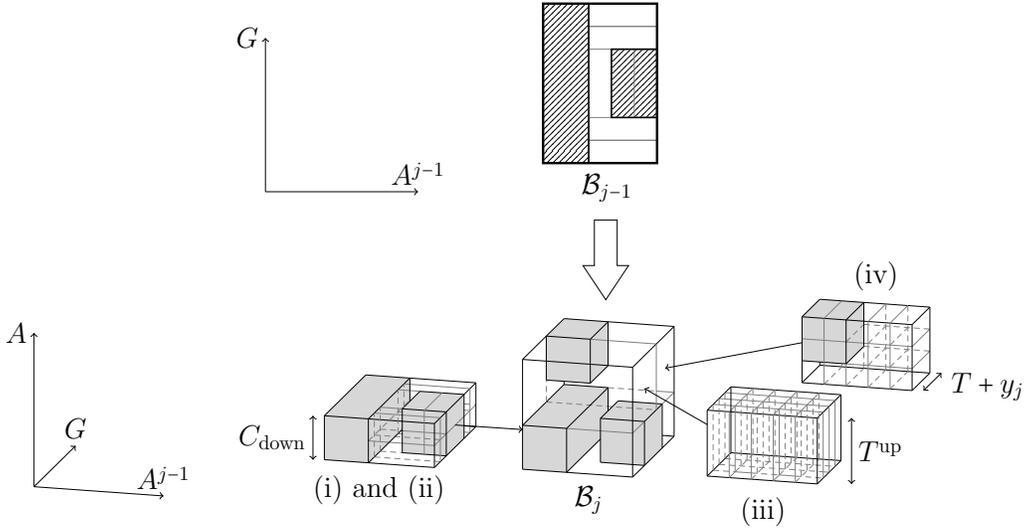}
                \caption{
                    The induction step in the
                    construction of the partition $\mathcal{B}$.
                } 
                \label{fig:blueprintconstruction}
            \end{figure}

            Define (recalling that $Y_0 = G \times C_{0,r}$
            and $Y_i = (T+y_i) \times C_{i,r}$ for $1 \le i
            \le t|G|$) 
            $$
                S =
                    \Omega \sm
                    \left(
                        \bigcup_{i = 0}^{t|G|} Y_i
                    \right).
            $$
            
            The point is that $S$ is tilable by the restriction of
            $\mathcal{B}$, and hence $S \times C_{0,m}$ is also 
            tilable. Therefore it only remains to prove that $M \sm
            \left(S \times C_{0,m} \right)$ is tilable, because
            this would imply that $M$ is tilable. Observe that $M \sm
            \left(S \times C_{0,m} \right) = (G \times A^d) \sm
            \uup[m]{S}$, so it remains to prove that $\uup[m]{S}$ is a hole.

            To prove this, fix any $g \in G$ and write $\Omega_g$ =
            $\{g\} \times A^r$. Then $\Omega_g$ intersects
            $Y_0$ and exactly $t|T| = m-1$ of the
            $Y_1, \dotsc, Y_{t|G|}$. In other words,
            $$
                \Omega_g \cap S =
                    \{g\} \times
                        \left(
                            A^r \sm
                            \bigcup_{k=1}^m C_{j_k,r}
                        \right)
            $$
            for some $0 = j_1 < j_2 < \dotsb < j_m \le r$. By
            Corollary~\ref{cor:mcsets}, $\uup[m]{(\Omega_g \cap S)}$
            is a hole in $\{g\} \times A^d$. This holds for
            any $g \in G$, so in fact $\uup[m]{S} = \bigcup_{g \in G}
            \uup[m]{(\Omega_g \cap S)}$ is a hole in
            $G \times A^d$, completing the proof.
        \end{proof}

    \subsection{Using one special dimension to cover certain subsets
        in slices}
    \label{subsec:specialdimension}

        In this section we show how one special dimension
        can be used to lay foundations for a tiling so that
        the tiling can be completed in each slice separately
        using Lemma~\ref{lem:constructdenser}. 

        Here is the main result of this section. Its statement
        and proof are very similar to Lemma~\ref{lem:specialdimension}
        from Section~\ref{sec:simplecase}.

        \begin{restatable}{lem}{coverholes}
        \label{lem:coverholes}
            Let $t, b \ge 1$ be integers and $T$ a finite
            tile in $\Z^b$. Further, let $S$ be a set and let $\mathcal{F}$
            be a family consisting of at least $(t-1)|T|^2$
            pairwise disjoint subsets of $S$. Then there is
            a set $X\subset\Z^b \times S$, satisfying:
            \begin{itemize}
                \item $X$ is a union of disjoint sets of the
                    form $(T+x)\times A$ with $x\in\Z^b$
                    and $A\in\mathcal{F}$, and
                \item for each $x\in\Z^b$ there is some $m \equiv
                    1 \pmod{t}$ such that
                    $
                        \{y \in S : (x, y) \in X\}
                    $
                    is a union of $m$ distinct members of
                    $\mathcal{F}$.
            \end{itemize}
        \end{restatable}

        We will deduce Lemma~\ref{lem:coverholes} from the following
        simple deconvolution type statement.

        \begin{prop}
            Let $t \ge 1$ and $b \ge 0$ be integers, $T$ a finite
            tile in $\Z^b$, and $f : \Z^b \to \Z$ a function. Then
            there is a function $g : \Z^b \to \{0, \dotsc, t-1\}$ such
            that for each $x \in \Z^b$
            $$
                \sum_{y \in T} g(x-y) \equiv f(x) \pmod{t}.
            $$
        \end{prop}
        
        \begin{proof}
            Use induction on $b$. The base case $b = 0$
            is trivial ($\Z^0$ being the trivial group),
            so suppose that $b \ge 1$. For any $n \in \Z$,
            write
            $$
                T_n = \{x \in \Z^{b-1} :
                    (x, n) \in T\}.
            $$

            Without loss of generality, assume that $T_0 \neq \es$
            and $T_n = \es$ for all $n < 0$. Write $k$ for the greatest
            integer such that $T_k \neq \es$. In other words, $[0, k]$
            is the minimal interval containing the projection of $T$
            in the last coordinate.

            Set $g(x, n)=0$ for all $x \in \Z^{b-1}$ and all $n \in \Z$
            such that $-k \le n \le -1$. The next step is to define
            $g(x, n)$ for all $n \ge 0$ and $x \in \Z^{b-1}$.
            Consider $N = 0, 1, \dotsc$ in turn, at each step having
            defined $g(x, n)$ whenever $-k \le n \le N - 1$
            and $x \in \Z^{b-1}$.  By the induction hypothesis,
            we can define $g(x, N)$ so that for all $x \in \Z^{b-1}$
            $$
                \sum_{y \in T_0} g(x - y, N) \equiv f(x, N) -
                    \sum_{
                        \mathclap{
                            \substack{
                                1 \le j \le k \\
                                z \in T_j
                            }
                        }
                    }
                    g(x-z, N-j) \hspace{10pt}
                (\text{mod } t).
            $$

            The final step is to define $g(x, n)$ when $n \le -k-1$.
            The argument is similar.
            Consider $N = -k-1, -k-2, \dotsc$ in turn, at each
            step having defined $f(x, n)$ whenever $n \ge N+1$. By
            the induction hypothesis, we can define $g(x, N)$
            so that for each $x \in \Z^{b-1}$
            $$
                \sum_{y \in T_k} g(x - y, N) \equiv f(x, N+k) -
                    \sum_{
                        \mathclap{
                            \substack{
                                0 \le j \le k-1 \\
                                z \in T_j
                            }
                        }
                    }
                    g(x-z, N+k-j) \hspace{10pt}
                (\text{mod } t).
            $$

            These three steps together define $g$ completely,
            and it is easy to see that $g$ satisfies the required condition.
        \end{proof}

        We get Lemma~\ref{lem:coverholes} as a quick corollary. In its
        proof we write $\N$ for $\{1, 2, \dotsc\}$.

        \begin{proof}[Proof of Lemma~\ref{lem:coverholes}]
            Let $g : \Z^b \to \{0,\dotsc,t-1\}$ be such that
            for each $x \in \Z^b$
            $$
                \sum_{y \in T} g(x - y) \equiv 1 \pmod{t}.
            $$
            
            Let $z_1, z_2, \dotsc$ be any enumeration of the elements
            of $\Z^b$.
            We will define sets $F_1, F_2, \dotsc \subset
            \mathcal{F}$ such that $|F_n| = g(z_n)$ for
            any $n \in \N$, and $F_m \cap F_n = \es$ for any
            distinct $m, n \in \N$ with $(T+z_m) \cap (T+z_n)
            \neq \es$. Then we will be done by taking
            $$
                X = \bigcup_{
                    \substack{
                        n \in \N \\
                        A \in F_n
                    }
                }
                \big[
                    (T + z_n) \times A
                \big].
            $$

            The $F_n$ can be defined inductively. Indeed, suppose
            that for some $N \in \N$ the sets $F_1, \dotsc, F_{N-1}$
            are already defined. Then we can define $F_N$ to consist
            of exactly $g(z_N)$ elements of $\mathcal{F}$ that
            are not contained in $F_n$ for any $n \le N-1$ with
            $(T + z_n) \cap (T + z_N) \neq \es$. This is possible,
            because we have at most $|T|^2 - 1$ choices for such
            $n$, and
            \[ 
                g(z_N) + (|T|^2 - 1) \max_{n \le N-1} |F_n|
                    \le (t-1) |T|^2 \le |\mathcal{F}|.
                \qedhere
            \]
        \end{proof}
        
    \subsection{General properties of tilings}
    \label{subsec:general}
        
        In this section we prove some transitivity
        results for tilings. The underlying theme is,
        expressed very roughly, `if $B$ is $A$-tilable
        with the help of $k$ extra dimensions,
        and $C$ is $B$-tilable with the help of $\ell$
        extra dimensions, then $C$ is $A$-tilable with the help
        of $k+\ell$ extra dimensions'.
        
        To avoid making the notation, which is already
        somewhat cumbersome, even more complicated we allow
        ourselves to abuse it in places where this is
        unlikely to create ambiguity. For example, given
        a tiling $X = \bigsqcup X_\alpha$ we may refer
        to the sets $X_\alpha$ as tiles (technically,
        they are not tiles, but copies of tiles). Otherwise,
        the proofs in this section are fairly straightforward.

        \newcommand{\is}{\;\text{ is }\;}

        \begin{prop}
        \label{prop:tilabletransitive}
            Let $G,\;G_1,\dotsc,G_m$ and $H_1,\dotsc,H_n$ be
            abelian groups with tiles $A\subset B\subset C
            \subset G,\; T_i\subset B_i\subset G_i$ and
            $U_i\subset C_i \subset H_i$. Suppose that
            \begin{alignat}{2}
                B_1 \times \dotsb \times B_m \times B
                    &\is&&
                (T_1,\dotsc,T_m, A)
                    \text{-tilable}
                    \label{cond:b-is-a-tilable}
            \shortintertext{and that}
                C_1 \times \dotsb \times C_n \times C^d
                    &\is&&
                (U_1, \dotsc, U_n, \multi[d]{B})
                    \text{-tilable.}
                    \label{cond:c-is-b-tilable}
            \shortintertext{Then}
                B_1 \times \dotsb \times B_m
                    \times C_1 \times \dotsb \times C_n \times C^d
                    &\is&&
                        (
                            T_1, \dotsc, T_m,
                            U_1, \dotsc, U_n, \multi[d]{A}
                        )
                    \text{-tilable}. 
                    \notag
            \end{alignat}
        \end{prop}

        Let us unravel the statement of this proposition. Intuitively,
        condition (\ref{cond:b-is-a-tilable}) asserts that `$B$ is almost
        $A$-tilable' -- the extra dimensions $B_1, \dotsc, B_m$ are used
        to fill the gaps. Similarly, condition (\ref{cond:c-is-b-tilable})
        asserts that `$C^d$ is almost $B$-tilable' -- here we use the
        extra dimensions $C_1, \dotsc, C_n$. Finally, the conclusion
        states that `$C^d$ is almost $A$-tilable' -- we use all the extra
        dimensions, $B_1, \dotsc, B_m$ and $C_1, \dotsc, C_n$, to complete
        this tiling.

        \begin{proof}
            For each tile $X$ in the $(U_1, \dotsc, U_n,\multi[d]{B})$%
            -tiling of $C_1 \times \dotsb \times C_n \times
            C^d$, partition the set $B_1 \times \dotsb \times
            B_m \times X$ in one of the two following ways:

            \begin{itemize}
                \item if $X$ is a copy of $B$, then partition
                    $B_1 \times \dotsc \times B_m \times X$
                    into its $(T_1, \dotsc, T_m,A)$-tiling;
                \item otherwise (that is, if $X$ is a copy
                    of one of the $U_1, \dotsc, U_n$),
                    partition the set into copies of $X$, namely
                    $\{b\} \times X$ for each $b \in
                    B_1 \times \dotsc \times B_n$.
            \end{itemize}

            This produces a $(T_1, \dotsc, T_m, U_1,
            \dotsc, U_n, \multi[d]{A})$-tiling of $B_1 \times \dotsb
            \times B_m \times C_1 \times \dotsb \times C_n
            \times C^d$.
        \end{proof}

        In the proof of the main theorem, we will apply this
        result in the following more compact form.

        \begin{cor}
        \label{cor:tilingtransitive}
            Let $G$ and $H$ be abelian groups with tiles $T \subset G$
            and $A \subset B \subset H$. Suppose that 
            \begin{alignat}{2}
                \label{eq:ghbtilable}
                G^k \times H^{\ell} \times B^d
                    &\is&&
                (\multi[k]{T}, \multi[(\ell+d)]A)
                    \text{-tilable}
            \shortintertext{and that}
                \label{eq:ghtilable}
                G^u \times H^v
                    &\is&&
                (\multi[u]{T}, \multi[v]{B})
                    \text{-tilable.}
            \shortintertext{Then}
                G^{du+k} \times H^{dv+\ell}
                    &\is&&
                (\multi[(du+k)]{T}, \multi[(dv+l)]{A})
                    \text{-tilable.}
                    \notag
            \end{alignat}
        \end{cor}

        \begin{proof}
            Use induction on $d$. The base case $d = 0$ is
            trivial, so suppose that $d \ge 1$.

            Rewrite ($\ref{eq:ghbtilable}$) to state that
            $$
                G^k \times H^l \times B^{d-1} \times B
                    \text{ is }
                (\multi[k]{T}, \multi[(\ell+d-1)]{A}, A)
                    \text{-tilable.}
            $$
            Now Proposition~\ref{prop:tilabletransitive}
            applied to this and (\ref{eq:ghtilable}) implies that
            $$
                G^k \times H^l \times B^{d-1}
                    \times G^u \times H^v
                    \text{ is }
                (\multi[k]{T}, \multi[(\ell+d-1)]{A},
                 \multi[u]{T}, \multi[v]{A})
                    \text{-tilable,}
            $$
            which after reordering and combining terms becomes the statement
            that
            $$
                G^{u+k} \times H^{v+l} \times B^{d-1}
                    \text{ is }
                (\multi[(u+k)]{T}, \multi[(v+l+d-1)]{A})
                    \text{-tilable.}
            $$
            Finally, apply the induction hypothesis to this and
            $(\ref{eq:ghtilable})$ to conclude the proof.
        \end{proof}

        The following straightforward proposition allows tilings to be lifted
        via surjective homomorphisms.

        \begin{prop}
        \label{prop:lifting}
            Let $G, H$ and $G_1, \dotsc, G_n$ be abelian groups
            with tiles $T \subset G$ and $U_i \subset G_i$,
            and let $\rho : G \to H$ be a surjective homomorphism
            that is injective on $T$. If 
            $G_1 \times \dotsb \times G_n \times H$ is
            $(U_1, \dotsc, U_n, \rho(T))$-tilable, then
            $G_1 \times \dotsb \times G_n \times G$ is
            $(U_1, \dotsc, U_n, T)$-tilable.
        \end{prop}

        \begin{proof}
            For any tile $X$ in the $(U_1, \dots, U_n,
            \rho(T)$)-tiling of $G_1 \times \dotsb \times G_n
            \times H$, let $\hat{X}$ denote the set
            $$
                \hat{X} =
                    \{
                        (x_1, \dotsc, x_n, x) \in
                        G_1 \times \dotsb \times G_n \times G :
                        (x_1, \dotsc, x_n, \rho(x)) \in X
                    \}.
            $$

            For every $X$, partition $\hat{X}$ in one of the two
            following ways:
            \begin{itemize}
                \item if $X$ is a copy of $U_i$ for some
                    $1 \le i \le n$, then partition $\hat{X}$
                    into copies of $U_i$ in the obvious way;
                \item if $X$ is a copy of $\rho(T)$, then
                    $X = \{(x_1, \dotsc, x_n)\} \times \rho(T+x)$
                    for some $x_i \in G_i$ and $x \in G$. Hence
                    $\hat{X} = \{(x_1, \dotsc, x_n)\} \times
                    (T + x + \ker(\rho))$, and as $\rho$ is
                    injective on $T$, this can be partitioned
                    into copies of $T$.
            \end{itemize}

            Since the sets $\hat{X}$ partition $G_1 \times
            \dotsb \times G_n \times G$, this produces a
            $(U_1, \dotsc, U_n, T)$-tiling for it.
        \end{proof}

        An inductive application of this proposition gives the following
        result, which we will use in the proof of the main theorem.

        \begin{cor}
        \label{cor:completelifting}
            Let $G$ and $H$ be abelian groups, and let $T \subset G$
            be a tile. Moreover, suppose that a surjective
            homomorphism $\rho : G \to H$ is injective on $T$.
            If $G^k \times H^\ell$ is $(\multi[k]{T}, \multi[\ell]%
            {\rho(T)})$-tilable, then $G^{k+\ell}$ is $T$-tilable.
            \qed
        \end{cor}

    \subsection{Proof of the main theorem}
    \label{subsec:formalproof}

        The tools needed for the proof Theorem~\ref{thm:main} are now
        available.

        \main*

        \begin{proof}
            Without loss of generality assume that $T
            \subset [k]^b$, where $k \in \N$. Write $G =
            \Zk{k}^b$ and let $\pi : \Z^b \to G$ be the
            projection map. In particular, $\pi(T)$
            is a tile in $G$.
            
            \begin{claim}
                Suppose that $A \subset G$ is a tile. Then
                there exist integers $p \ge 0$ and $q \ge 1$ such that
                $(\Z^b)^p \times G^q$ is $(\multi[p]{T},
                \multi[q]{A})$-tilable.
            \end{claim}
            
            \begin{claimproof}
                Use reverse induction on $|A|$.
                If $|A| = |G|$ then in fact $A = G$,
                and the claim holds with $p = 0$, $q = 1$.
                So suppose that $|A| \le |G|-1$.

                Applying Lemma~\ref{lem:constructdenser} to
                the tile $A$ with fixed large $d_0$ produces a number
                $d_1 \ge d_0$, a set $B$ such that $A \subsetneq B
                \subset G$ and a family $\mathcal{F}$
                ($|\mathcal{F}| \ge d_0$) of pairwise disjoint
                subsets of $B^{d_1}$ with the property that for
                any subfamily $\mathcal{S} \subset \mathcal{F}$ of size
                satisfying $|\mathcal{S}| \equiv 1\pmod{|A|}$, the set
                $$
                    G \times
                        \left(
                            B^{d_1} \sm \bigcup_{S\in\mathcal{S}}
                            S
                        \right)
                $$
                is $A$-tilable.

                Since $d_0$ is large, Lemma~\ref{lem:coverholes}
                gives a set $X \subset \Z^b \times G \times B^{d_1}$
                that is a disjoint union of copies of $T$, and such
                that for every $x \in \Z^b$ the slice $\{y \in G\times
                B^{d_1} : (x, y) \in X\}$ is a hole in $G \times B^{d_1}$.
                Therefore $\Z^b \times G \times B^{d_1}$ is
                $(T, \multi[(d_1+1)]{A})$-tilable.

                By the induction hypothesis, there exist $u \ge 0$
                and $v \ge 1$ such that $(\Z^b)^{u} \times
                G^{v}$ is ${(\multi[u]{T}, \multi[v]{B})}$-tilable.
                Now apply Corollary~\ref{cor:tilingtransitive}
                to conclude that the claim holds with
                $p = d_{1}u+1$ and $q = d_{1}v+1$. This proves the claim.
            \end{claimproof}

            To finish the proof of the theorem, apply the claim
            to the tile $\pi(T)$.
            This gives $p \ge 0$ and $q \ge 1$
            such that $(\Z^b)^p \times G^q$ is $(\multi[p]{T},
            \multi[q]{\pi(T)})$-tilable. Hence, by Corollary~%
            \ref{cor:completelifting}, $(\Z^b)^{p+q}$ is $T$-tilable.
        \end{proof}

    \section{Concluding remarks and open problems}
    \label{sec:conclusion}

    We mention in passing that all our tilings are (or can be made to be)
    periodic. Also, our copies of $T$ arise only from translations and
    permutations of the coordinates -- in particular, `positive directions
    stay positive'.

    We have made no attempt to optimise the dimension $d$ in Theorem~%
    \ref{thm:main}. What can be read out of the proof is the following.

    \begin{thmprimed}[\ref{thm:main}]
        Let $T \subset \Z^n$ be a tile and suppose that $T \subset [k]^n$.
        Then $T$ tiles $\Z^d$, where $d = \lceil \exp(100 (n \log k)^2)
        \rceil$.
    \end{thmprimed}

    Thus our upper bound on $d$ is superpolynomial in the variable $k^n$.
    We believe that there should be an upper bound on $d$ in terms only
    of the size and dimension of $T$. Even in the case $n=1$ this seems
    to be a highly non-trivial question.

    \begin{conj}
        For any positive integer $t$ there is a number $d$ such that any
        tile $T \subset \Z$ with $|T| \le t$ tiles $\Z^d$.
    \end{conj}

    On the other hand, it is easy to see that there cannot be a bound just
    in terms of the dimension
    of the tile. Indeed, given any $d$ it is possible to find a one-dimensional
    tile that does not tile $\Z^d$. Such a tile $T$ can be constructed by
    fixing an integer $k$ and taking two intervals of length
    $k$, distance $k^2 - 1$ apart, where in between the intervals only every
    $k$-th point is present in the tile. For example, if $k = 4$ then
    the resulting tile would be
    $$
        \mathtt{XXXX...X...X...X...XXXX}
    $$

    Suppose that $T$ tiles $\Z^d$. Choose a large integer $N$ and 
    consider the cuboid $[N]^d$. Fix one of the $d$ directions and only
    consider the copies of $T$ in this direction that intersect $[N]^d$.
    There can be at most $N^{d-1} (N/(k^2 + 2k - 1) + 2) = O(N^d / k^2)$
    such tiles and they cover at most $O(N^d / k)$ elements of $[N]^d$.
    Since there are $d$ directions, at most $O(dN^d / k)$ elements of
    the cuboid
    can be covered with tiles, but this number is less than $N^d$ for
    large $k$. Therefore, if $k$ is large enough, then $T$ does not tile
    $\Z^d$.

    Finally, we do not know how to find reasonable lower bounds, even for
    seemingly simple tiles. In particular, we do not know the smallest
    dimension that can be tiled by a given interval with the central
    point removed.  The largest tile of this shape for which we know
    the answer is $\mathtt{XXX.XXX}$: it does not tile $\Z^2$ (by case
    analysis), but it tiles $\Z^3$. (One way to achieve this is to adapt
    the argument of Section~\ref{sec:simplecase}, to make it work in
    the case of this tile for $d=3$.)

    \begin{qn}
        Let $T$ be the tile
        $
        \underbrace{\mathtt{XXXXX}}_k
        .
        \underbrace{\mathtt{XXXXX}}_k
        $ .
        What is the least $d$ for which $T$ tiles $\Z^d$?
    \end{qn}
    
    Just getting rough bounds on $d$ would be very interesting. We do
    not even know if $d \to \infty$ as $k \to \infty$.

    \renewcommand{\biblistfont}{\normalfont\normalsize}

    \bibliography{tilings}

\end{document}